\documentclass{proc-l}
\usepackage{relsize}
\usepackage{enumerate}
\usepackage{graphicx}
\usepackage{tikz}
\usetikzlibrary{positioning}
\usepackage{tikz-cd}
\usepackage{color}
\usepackage{tabularx}
\usepackage{ragged2e}
\usepackage{enumitem}
\usepackage{ circuitikz }
\newcolumntype{C}{>{\Centering\arraybackslash}X}

\DeclareMathAlphabet\mathfrak{U}{euf}{m}{n}
\SetMathAlphabet\mathfrak{bold}{U}{euf}{b}{n}

\usepackage{hyperref}
\usepackage{xcolor}
\hypersetup{
    colorlinks,
    linkcolor={blue!50!black},
    citecolor={blue!50!black},
    urlcolor={blue!80!black}
}
\usepackage[nameinlink]{cleveref}
\newtheorem{thm}{Theorem}
\newtheorem{lemma}[thm]{Lemma}
\newtheorem{cor}[thm]{Corollary}

\theoremstyle{definition}
\newtheorem{defi}[thm]{Definition}
\newtheorem{ex}[thm]{Example}

\newtheorem{remark}[thm]{Remark}
\numberwithin{equation}{section}
\crefname{defi}{definition}{definitions}
\Crefname{defi}{Definition}{Definitions}
\crefname{thm}{theorem}{theorems}
\Crefname{thm}{Theorem}{Theorems}

\newtheorem*{thm*}{Theorem}

\newcommand{\N}{\mathbb{N}}

\newcommand{\cb}[1]{ \big \{ #1 \big \} }

\newcommand\restr[2]{{\left.\kern-\nulldelimiterspace#1 \vphantom{\big|}\right|_{\mathlarger{#2}}}}

\begin{document}
    \title[Bounds on the Threshold Ramsey Multiplicity of Ramsey Numbers]{Bounds on the Threshold Ramsey Multiplicity of Ramsey Numbers with Many Colors}
    \author[B. Christopherson]{Bryce Alan Christopherson$^1$}
    \address{Department of Mathematics, University of North Dakota, Grand Forks, ND 58202}
    \email{bryce.christopherson@UND.edu}
    
    \author[C. Steinhaus]{Casia Steinhaus}
    \address{University of North Dakota, Grand Forks, ND 58202}
    \email{casia.steinhaus@UND.edu}
    

    
    
    \begin{abstract}
    
        The \textit{Ramsey number} $R(s,t)$ is the least integer $n$ such that any coloring of the edges of $K_n$ with two colors produces either a monochromatic $K_s$ in one color or a monochromatic $K_t$ in the other.  If $s=t$, we say that the Ramsey number $R(s,s)$ is \textit{diagonal}.  The \textit{threshold Ramsey multiplicity} of a diagonal Ramsey number $R(s,s)$, denoted $m(s,s)$ or $m_2(s)$, is the smallest number of copies of a monochromatic $K_s$ that can be found in any coloring of the edges of $K_{R(s,s)}$.  For instance, $m_2(2)=1$, $m_2(3)=2$, and $m_2(4)=9$. We derive upper bounds for multicolor, off-diagonal threshold Ramsey multiplicities. In the diagonal two-color case, the resulting bounds improve the elementary random-coloring estimate for $5\leq s\leq8$. In particular, we recover the known value $m(3,3,3)=5$ and obtain the bound $m(3,3,4)\leq10$. We conclude with a general framework for seeking further improvements.
    \end{abstract}

    \maketitle
\footnotetext[1]{Corresponding Author}
\section*{Introduction}

The \textit{Ramsey number} $R(s,t)$ is the least integer $n$ such that any coloring of the edges of $K_n$ with two colors produces either a monochromatic $K_s$ in one color or a monochromatic $K_t$ in the other. For a positive integer $k$, write $[k]=\{1,\ldots,k\}$. More generally, we have the following definition.

    \begin{defi}[Ramsey Numbers]\label{ramsey def}
    
        Let $K_n$ denote the undirected complete graph on $n$ vertices, and let $s_1,\hdots,s_k \in \N$.  The \textit{Ramsey number} $R(s_1,\hdots,s_k)$ is the smallest $n \in \N$ such that, for any $k$-coloring of the edges of $K_n$, there exists a subgraph of $K_n$ which is monochromatic in color $i$ and isomorphic to $K_{s_i}$ for some $i \in \cb{1,\hdots,k}$.
        
    \end{defi}
    
    Ramsey numbers $R(s_1,\hdots,s_k)$ for which $s_1=\hdots=s_k$ are often called \textit{diagonal} and are denoted by $R_k(s)$ for convenience. The so-called \textit{threshold Ramsey multiplicity} of a Ramsey number describes the minimum number of monochromatic copies of $K_{s_i}$ in any $k$-coloring of the edges of $K_{R(s_1,\hdots,s_k)}$.  This threshold version of Ramsey multiplicity was first studied systematically by Harary and Prins \cite{HararyPrins1974}; see also the early survey of Burr and Rosta \cite{BurrRosta1980}.  We use the following multicolor, off-diagonal definition.
    
    \begin{defi}[Threshold Ramsey Multiplicity]\label{crit mult def}
    
        The \textit{threshold Ramsey multiplicity} of the Ramsey number $R(s_1,\hdots,s_k)$, denoted $m(s_1,\hdots,s_k)$, is the minimum, over all $k$-colorings of the edges of $K_{R(s_1,\hdots,s_k)}$, of the total number of monochromatic copies of $K_{s_i}$ in color $i$, summed over all colors.
        
    \end{defi}
    
    Occasionally, $m(s_1,\hdots,s_k)$ is also called the \textit{critical multiplicity}, as was suggested by Conlon in a communication to Caicedo \cite{28132}. This terminology is less common, perhaps because of the potential for confusion with a related concept also called Ramsey multiplicity \cite{gasarch2020open}, which we will not address.
    
    Comparatively little is known for complete-graph targets, even in the two-color case.  For two diagonal colors, $m_2(2)=1$, $m_2(3)=2$, and $m_2(4)=9$; the last value was determined computationally by Piwakowski and Radziszowski \cite{PiwakowskiRadziszowski2001}.  Off-diagonal threshold instances were considered earlier as well: $m(3,4)=1$ \cite{KhadzhiivanovNenov1978}, $m(3,5)=4$ \cite{Pashov1984}, and $m(3,6)=2$ \cite{KhadshiivanovPashov1992}. In the three-color case, Sane and Wallis proved that $m(3,3,3)=5$ \cite{SaneWallis1988}. The elementary probabilistic argument \cite{BurrRosta1980} yields the bound $m_2(s)\leq {R_2(s) \choose s}2^{1-{s \choose 2}}$.  No useful general lower bounds are known for complete-graph targets; indeed, whether $m_2(s)$ increases monotonically with $s$ is an open question \cite{26040}.  Recent work has also investigated threshold Ramsey multiplicity for paths and cycles in the two-color setting \cite{ConlonFoxSudakovWeiEven,ConlonFoxSudakovWeiOdd}, though we will not address this here, as our focus is instead on complete-graph targets with arbitrarily many colors and unequal target sizes.

    The principal concrete bounds are given in \Cref{bound theorem,formula theorem2}. In the diagonal two-color case, \Cref{bound theorem} improves the elementary probabilistic estimate for $5\leq s\leq8$. In the multicolor case, \Cref{formula theorem2} recovers the known value $m(3,3,3)=5$ and yields $m(3,3,4)\leq10$. Finally, \Cref{bound theorem m version} extends the counting argument to partial colorings that are cleanly saturated with respect to a general uncolored subgraph, providing a framework for seeking further bounds.

\section{An Extension Argument}\label{extend}

    In the first part of our argument, we will consider \textit{extending} a coloring of a subgraph $H$ of a graph $G$ to a coloring of $G$ itself.  
    
    \begin{defi}[Extension of a coloring]\label{graph coloring}
    
        For a subgraph $H \leq G$ and some set $A$ of colors, we say that a coloring $c^\prime: E(G) \rightarrow A$ is an \textit{extension} of a coloring $c: E(H) \rightarrow A$ if $\restr{c^\prime}{E(H)}=c$.
        
    \end{defi}
    
    When extending a $k$-coloring of the edges of $K_{R(s_1,\hdots,s_k)-1}$ which is free of any monochromatic subgraph isomorphic to $K_{s_j}$ in color $j$ for all $j \in \left\{1,\hdots,k\right\}$ to a $k$-edge-coloring of $K_{R(s_1,\hdots,s_k)}-e$, it is always possible to do so in such a way that one avoids a monochromatic subgraph isomorphic to $K_{s_j}$ in color $j$ for all $j \in \left\{1,\hdots,k\right\}$ precisely until the color of the final edge is decided.  The argument is essentially just the multicolor version of the cloning argument used by Jacobson in \cite[Theorem 1]{jacobson1980note}.
    
    \begin{lemma}\label{last edge lemma}
    
        For $n \geq 2$, let $K_n-e$ be the graph isomorphic to $K_n$ with a single edge removed.  If $n-1<R(s_1,\hdots,s_k)$, then there is a $k$-coloring of the edges of $K_n-e$ which is free of any monochromatic subgraph isomorphic to $K_{s_j}$ in color $j$ for all $j \in \left\{1,\hdots,k\right\}$.
        
    \end{lemma}
    \begin{proof}
    
        By assumption, $n-1<R(s_1,\hdots,s_k)$.  So, there exists a $k$-coloring of the edges of $K_{n-1}$ free of any monochromatic subgraph isomorphic to $K_{s_j}$ for each $j \in \left\{1,\hdots,k\right\}$.  Take such a coloring of $K_{n-1}$, choose a vertex $v \in V(K_{n-1})$, and add an additional vertex $*$ and $n-2$ new edges between $V(K_{n-1})\setminus \left\{v\right\}$ and $*$ to form a graph $G$ isomorphic to $K_n-e$, i.e. so that $V(G)=V(K_{n-1})\cup \left\{*\right\}$ and $E(G)=E(K_n)\setminus \left\{e\right\}$, where $e$ is the edge between $v$ and $*$.  Extend the coloring of $K_{n-1}$ to $G$ by assigning each edge $*w$ the color of the edge $vw$ in the original coloring.  Since both the subgraph induced by $V(G)\setminus\{v\}$ and the original $K_{n-1}$ are free of any monochromatic subgraph isomorphic to $K_{s_j}$ for each $j \in \{1,\hdots,k\}$, there is no monochromatic subgraph $H$ of $G$ isomorphic to $K_{s_j}$ in color $j$ whose vertex set $V(H)$ satisfies either $V(H) \subseteq V(G)\setminus\left\{v\right\}$ or $V(H) \subseteq V(G)\setminus \left\{*\right\}$ for any $j \in \{1,\hdots,k\}$.  So, if such a monochromatic subgraph $H$ exists within our edge coloring of $G$, then $\left\{v,*\right\}\subseteq V(H)$.  But there is no edge between $v$ and $*$, so such a subgraph cannot exist.
        
    \end{proof}
    \begin{remark}\label{rem:two}
    
        The preceding definitions also cover the degenerate cases. For one
        color, $R(s)=s$ and $m(s)=1$. If $s_i=1$ for at least one $i\in[k]$,
        then
        $$
            R(s_1,\ldots,s_k)=1
            \qquad\text{and}\qquad
            m(s_1,\ldots,s_k)
            =
            \left|\{i\in[k]:s_i=1\}\right|.
        $$
        If all $s_i\geq2$ and at least one $s_j=2$, then
        $m(s_1,\ldots,s_k)=1$. We therefore assume that $s_i\geq3$
        for every $i\in[k]$ from here onward.
        
    \end{remark}
    
\section{A Counting Argument}\label{counting}

    In this section we will first produce a simple counting argument, in two parts.  In the first part, we will prove a lemma about graphs containing many copies of a complete graph glued together along a common edge, which says that the number of copies is bounded by a particular function of the number of vertices (\Cref{counting inequality}).  Then, in the second part, we will use this in conjunction with the extension argument from before (\Cref{last edge lemma}) to bound the threshold Ramsey multiplicity (\Cref{bound theorem}).  This bound is better than the standard one obtained by the probabilistic method for small Ramsey numbers, but rapidly becomes worse for larger ones.  A slightly refined argument (\Cref{formula theorem}) offers room for further improvements.
    
    \begin{lemma}\label{counting inequality}
    
        Let $s \geq 3$ and let $H$ be a graph with $|V(H)|\geq s$ such that (i) $H$ contains $t$ distinct copies of $K_s$ (i.e. subgraphs $H_1,\hdots,H_t$ each isomorphic to $K_s$ such that $H_i \neq H_j$ for $i \neq j$); and (ii) $H$ contains an edge $e \in E(H)$ such that $H-e$ contains zero copies of $K_s$.  Then, $$t \leq \prod_{j=0}^{s-3}\left\lceil \frac{\big(|V(H)|-2\big)-j}{s-2}\right\rceil.$$ Moreover, for every prescribed value of $|V(H)|\geq s$, there is a graph $H$ with these properties that achieves equality in the displayed bound.
        
    \end{lemma}
    \begin{proof}
    
        First, we construct such a graph $H$ with these properties that achieves this value. Fix $N\geq s$, take a set of $N$ vertices, distinguish two of them as $u$ and $v$, and let $e=uv$.  Partition the remaining $N-2$ vertices into $s-2$ subsets $A_0,\hdots,A_{s-3}$ such that 
        
        $$
            |A_j|=\left\lceil\frac{(N-2)-j}{s-2}\right\rceil.
        $$
        
        Now, recalling that $n=\sum_{j=0}^{m-1}\left\lceil\frac{n-j}{m}\right \rceil$ for any positive integer $m$, we have $\sum_{j=0}^{s-3}|A_j|=N-2$, so this is indeed a partition.  Form $H$ by joining the edge $uv$ to the complete $(s-2)$-partite graph with parts $A_0,\hdots,A_{s-3}$. Thus, every copy of $K_s$ in $H$ consists of $u,v$ and one vertex from each subset $A_0,\hdots,A_{s-3}$; in particular, $H-e$ contains no copy of $K_s$.  There are $\prod_{j=0}^{s-3}|A_j|$ such choices, producing the same number of copies $t$ of $K_s$, all sharing the common edge $e$.  Since $N=|V(H)|$, $H$ achieves the bound, as 
        
        $$
            t=\prod_{j=0}^{s-3}|A_j|=\prod_{j=0}^{s-3}\left\lceil\frac{\big(|V(H)|-2\big)-j}{s-2}\right\rceil.
        $$

        Now, suppose $H$ is any graph with $t$ subgraphs isomorphic to $K_s$ and an edge $e\in E(H)$ such that $H-e$ has zero copies of $K_s$.  Let $u$ and $v$ be the vertices of $H$ incident to $e$, let $C\subseteq V(H)\setminus\{u,v\}$ denote the subset of vertices adjacent to both $u$ and $v$ in $H$, and let $H[C]$ denote the subgraph of $H$ induced by $C$.  Then, there is a bijection between the copies of $K_s$ in $H$ and the copies of $K_{s-2}$ in $H[C]$.  Moreover, $H[C]$ is $K_{s-1}$-free, since a copy of $K_{s-1}$ in $H[C]$, together with $u$ (or $v$, if you prefer), would be a copy of $K_s$ in $H-e$.  Accordingly, since $H[C]$ is a graph on at most $|V(H)|-2$ vertices with no $K_{s-1}$, Zykov's theorem \cite{Zykov1949} applies and $H[C]$ must contain at most the number of $K_{s-2}$'s in the Tur\'an graph $T(|V(H[C])|,s-2)$.  But the number of $K_{s-2}$'s in $H[C]$ is $t$, so we get
        $$
            t \leq \prod_{j=0}^{s-3}\left\lceil\frac{|V(H[C])|-j}{s-2}\right\rceil\leq \prod_{j=0}^{s-3}\left\lceil
            \frac{\bigl(|V(H)|-2\bigr)-j}{s-2}\right\rceil.
        $$
        
    \end{proof}
    
    Using this with \Cref{last edge lemma}, we get the following:
    
    \begin{thm}\label{bound theorem}
    
        \begin{equation}\label{ourbound}
        \displaystyle m(s_1,\hdots,s_k) \leq \max_{1\leq i\leq k}\prod_{j=0}^{s_i-3}\left\lceil \frac{\left\lfloor\frac{R(s_1,\hdots,s_k)-2}{k}\right\rfloor-j}{s_i-2}\right\rceil
        \end{equation}
        
    \end{thm}
    \begin{proof}
    
        Write $R(s_1,\hdots,s_k)=n$.  We know there exists a $k$-coloring of the edges of $K_{n-1}$ which avoids a monochromatic $K_{s_i}$ for all $i \in \left\{1,\hdots,k\right\}$.  Use the extension argument in \Cref{last edge lemma} to extend this to a coloring of $K_{n}-e$ which avoids a monochromatic $K_{s_i}$ for all $i \in \left\{1,\hdots,k\right\}$ as well.  Consider the $k$ possible extensions of this coloring. For each $i\in\{1,\hdots,k\}$, let $G_i$ be the union of the monochromatic color-$i$ copies of $K_{s_i}$ formed when the final edge is assigned color $i$. No target in a color other than $i$ can occur in this extension, since it would already occur in the target-free coloring of $K_n-e$. Hence, each $G_i$ has at least $m(s_1,\hdots,s_k)$ copies of $K_{s_i}$ sharing the final edge, the deletion of which leaves $G_i$ with none. So, by \Cref{counting inequality},
        $$
            m(s_1,\hdots,s_k)\leq \prod_{j=0}^{s_i-3}\left\lceil \frac{\big(|V(G_i)|-2\big)-j}{s_i-2}\right\rceil
        $$
        for $i=1,\hdots,k$.  No vertex other than the two endpoints of the final edge can belong to both $G_i$ and $G_j$ for $i\neq j$, since its edges to either endpoint would then have both colors $i$ and $j$. Thus,
        $$
            n\geq 2+\sum_{i=1}^k\big(|V(G_i)| -2\big).
        $$
        Choose $i^*\in\operatorname*{argmin}_{1\leq i\leq k}|V(G_i)|$.  Then, $n \geq 2+k\big(|V(G_{i^*})| -2\big)$.  Through some minor algebra, this yields
        
            \begin{align*}
                R(s_1,\hdots,s_k) &\geq k\big(|V(G_{i^*})| -2\big)+2\\
                \frac{R(s_1,\hdots,s_k)-2}{k}&\geq \big(|V(G_{i^*})| -2\big)\\
                \left\lfloor\frac{R(s_1,\hdots,s_k)-2}{k}\right\rfloor&\geq \big(|V(G_{i^*})| -2\big)\\
                \max_{1\leq t\leq k}\prod_{j=0}^{s_{t}-3}\left\lceil \frac{\left\lfloor\frac{R(s_1,\hdots,s_k)-2}{k}\right\rfloor-j}{ s_{t}-2}\right\rceil &\geq\prod_{j=0}^{s_{i^*}-3}\left\lceil\frac{\big(|V(G_{i^*})| -2\big)-j}{s_{i^*}-2}\right\rceil \geq m(s_1,\hdots,s_k).
            \end{align*}
            
    \end{proof}
    
    Using \Cref{bound theorem}, we get $m(3,3) \leq 2$ and $m(4,4) \leq 16$, so the bound is not always tight.  It is easy to check that this is an improvement on the naive bound
    \begin{equation}\label{classicbound}
        m(s,s)\leq \frac{{R(s,s) \choose s}}{2^{{s\choose 2}-1}}
    \end{equation}
    for small $s$ (in particular, for $5\leq s \leq 8$; see \Cref{table1}), though it rapidly becomes worse for larger $s$. The same argument also applies in the off-diagonal and multicolor settings.
    
    \begin{table}[hpt!]
    
        \centering
        \small
        
        \begin{tabular}{|c|c|c|c|}\hline
            s   & $R_2(s)$             &Classical Bound\eqref{classicbound}            & Our Bound \eqref{ourbound}                     \\\hline\hline
            $5$ & $R_2(5) \leq 46$     & $m_2(5)\leq 2,677$                          & $m_2(5)\leq 392$                               \\\hline
            $6$ & $R_2(6)\leq 160$     & $m_2(6)\leq 1,293,533$                      & $m_2(6)\leq 152,000$                            \\\hline
            $7$ & $R_2(7)\leq 492$     & $m_2(7)\leq 1,265,045,472$                  & $m_2(7)\leq 282,475,249$                     \\\hline
            $8$ & $R_2(8)\leq 1,518$   & $m_2(8)\leq 5,114,696,152,715$              & $m_2(8)\leq 4,065,272,127,504$                 \\\hline
            $9$ & $R_2(9)\leq 4,956$   & $m_2(9)\leq 1.437\times 10^{17}$            & $m_2(9)\leq 6.947\times 10^{17}$               \\\hline
            $10$& $R_2(10)\leq 16,064$ & $m_2(10)\leq 1.788\times 10^{22}$           & $m_2(10)\leq 1.032 \times 10^{24}$              \\\hline
        \end{tabular}
        
        \caption{Values of the bound given by \Cref{bound theorem}, \eqref{ourbound}, versus the classical probabilistic bound \eqref{classicbound}, using the Ramsey-number upper bounds in \cite{Radziszowski2026}.  The bound \eqref{ourbound} is an improvement on \eqref{classicbound} for $5\leq s\leq 8$, but is worse in the two subsequent cases shown.}
        
        \label{table1}
        
    \end{table}
    
    The argument in \Cref{bound theorem} can be refined slightly.  We need a lemma to do this.
    
    \begin{lemma}\label{size lemma}
    
         Write $R(s,s)=n$.  Given an edge-coloring of $K_{n}-e$ which avoids a monochromatic $K_{s}$, consider the $2$ possible extensions of this coloring and let $G_1,G_2$ be the corresponding subgraphs of $K_{n}$ formed by choosing a color for this final edge and taking only the vertices in the monochromatic $K_s$'s containing that edge and the corresponding edges in that color incident to those vertices.  Then, $\bigl||V(G_1)|-|V(G_2)|\bigr| < R(s-1,s)-(s-2)$.
         
    \end{lemma}
    \begin{proof}
    
        For each $i\in\{1,2\}$, let $S_i$ be the union of the vertex sets of the color-$i$ copies of $K_s$ containing the final edge in the $i$th extension, and let $G_i$ be the graph on $S_i$ whose edges are all color-$i$ edges with both endpoints in $S_i$.  Without loss of generality, assume $|V(G_1)|\geq|V(G_2)|$.  We proceed by way of contradiction and suppose $|V(G_1)|-|V(G_2)|\geq R(s-1,s)-(s-2)$.  Since $|V(G_2)| \geq s$ (as the contrary would yield that $R(s,s)>n$), this gives $|V(G_1)| \geq R(s-1,s)+2$.  That is, if $v_1,v_2 \in V(G_1)\cap V(G_2)$ are the two vertices incident to the distinguished edge $e$, then $|V(G_1)-\{v_1,v_2\}| \geq R(s-1,s)$.  Accordingly, either $G_1-\{v_1,v_2\}$ contains a monochromatic $K_{s-1}$ or its edge-wise complement contains a $K_{s}$.  But if $G_1-\{v_1,v_2\}$ contains a monochromatic $K_{s-1}$, then $G_1-v_1$ (as well as $G_1-v_2$, if you prefer) contains a monochromatic $K_s$, contradicting our initial assumption that $K_n-e$ is monochromatic $K_s$-free.  Likewise, if the edge-wise complement of $G_1-\{v_1,v_2\}$ contains a $K_{s}$, all of these edges must be in the other color and, again, our assumption that $K_n-e$ is monochromatic $K_s$-free is violated.
        
    \end{proof}
    
    This produces the following alternative to \Cref{bound theorem}, the proof of which follows exactly along the same lines, imposing only the conditions of \Cref{size lemma}.
    
    \begin{thm}\label{formula theorem}
    
        \begin{equation}\label{ourbound2}
            m_2(s) \leq \max\left\{\min_{i \in \{1,2\}}\left\{\prod_{j=0}^{s-3}\left\lceil\frac{\big(n_i-2\big)-j}{s-2}\right\rceil\right\}\right\}
        \end{equation}
        
        where the maximum is taken over $n_1,n_2 \in \N$ satisfying the following:
        
        \begin{enumerate}[wide, labelindent=0pt]
            \item $n_1,n_2 \geq s$, 
            \item $n_1+n_2\leq R_2(s)+2$, 
            \item $|n_1-n_2|< R(s-1,s)-(s-2)$.
        \end{enumerate}
        
    \end{thm}
    \begin{proof}
    
        For the two extensions considered in \Cref{size lemma}, set
        $n_i=|V(G_i)|$. Conditions (1) and (2) follow from the construction
        and the overlap argument in the proof of \Cref{bound theorem},
        while condition (3) is precisely \Cref{size lemma}. Applying
        \Cref{counting inequality} to each $G_i$, taking the minimum over
        $i$, and then maximizing over all feasible tuples gives the result.
        
    \end{proof}
    \begin{ex}
    
        Using $R(4,5)=25$ and $R(5,5)\leq 46$ with \eqref{ourbound2}, we get
        $$
            \displaystyle m_2(5) \leq
            \max_{\substack{
                n_1,n_2\in \N,\enskip n_1+n_2\leq 48,\\
                n_1,n_2 \geq 5,\enskip |n_1-n_2|<22
            }}
            \left\{
                \min_{i \in \{1,2\}}
                \left\{
                    \prod_{j=0}^{2}
                    \left\lceil\frac{\big(n_i-2\big)-j}{3}\right\rceil
                \right\}
            \right\}.
        $$
        So, $m_2(5) \leq 392$, which is much better than the classical bound \eqref{classicbound} of $m_2(5) \leq 2,677$ and agrees with the bound provided by \eqref{ourbound}.
        
    \end{ex}
    It is possible to extend this to the multicolor, off-diagonal case as well by extending \Cref{size lemma} as follows.
    
    \begin{lemma}\label{size lemma 2}
        
        Suppose $k \geq 2$ and write $R(s_1,\hdots,s_k)=n$.  Given an edge-coloring of $K_{n}-e$ which avoids a monochromatic $K_{s_t}$ for $t \in \{1,\hdots,k\}$, consider the $k$ possible extensions of this coloring and let $G_i$ be the union of the color-$i$ copies of $K_{s_i}$ containing the final edge in the $i$th extension and the corresponding edges in that color incident to those vertices.  Then, 
        $$
        |V(G_i)|-|V(G_j)|< R\big(s_i-1,R(s_1,\hdots,s_{i-1},s_{i+1},\hdots,s_k)\big)-(s_j-2)
        $$ 
        for all distinct $i,j \in \{1,\hdots,k\}$ with $|V(G_i)|\geq|V(G_j)|$.
        
    \end{lemma}
    \begin{proof}
    
        Write $R(s_1,\hdots,s_k)=n$.  Given an edge-coloring of $K_{n}-e$ which avoids a monochromatic $K_{s_t}$ for $t \in \{1,\hdots,k\}$ (as is possible by \Cref{last edge lemma}), consider the $k$ possible extensions of this coloring and, for each $i\in[k]$, let $S_i$ be the union of the vertex sets of the color-$i$ copies of $K_{s_i}$ containing the final edge in the $i$th extension, and let $G_i$ be the graph on $S_i$ whose edges are all color-$i$ edges with both endpoints in $S_i$. Without loss of generality, assume $|V(G_i)|\geq|V(G_j)|$.  We proceed by way of contradiction and suppose $|V(G_i)|-|V(G_j)|\geq R\big(s_i-1,R(s_1,\hdots,s_{i-1},s_{i+1},\hdots,s_k)\big)-(s_j-2)$.  Since $|V(G_j)| \geq s_j$ (as the contrary would yield that $R(s_1,\hdots,s_k)>n$), this gives $|V(G_i)| \geq R\big(s_i-1,R(s_1,\hdots,s_{i-1},s_{i+1},\hdots,s_k)\big)+2$.  That is, if $v_1,v_2 \in V(G_i)\cap V(G_j)$ are the two vertices incident to the distinguished edge $e$, then $|V(G_i)-\{v_1,v_2\}| \geq R\big(s_i-1,R(s_1,\hdots,s_{i-1},s_{i+1},\hdots,s_k)\big)$.  Accordingly, either $G_i-\{v_1,v_2\}$ contains a monochromatic $K_{s_i-1}$ or, with the original colors on its edges retained, its edge-wise complement contains a monochromatic $K_{s_t}$ for some $t\in[k]\setminus\{i\}$.  But if $G_i-\{v_1,v_2\}$ contains a monochromatic $K_{s_i-1}$, then $G_i-v_1$ (as well as $G_i-v_2$, if you prefer) contains a monochromatic $K_{s_i}$, contradicting our initial assumption that $K_n-e$ contains no color-$t$ copy of $K_{s_t}$ for any $t\in\{1,\hdots,k\}$.  Likewise, if the edge-wise complement of $G_i-\{v_1,v_2\}$ contains a monochromatic $K_{s_t}$ for some $t\in \{1,\hdots,k\}\setminus\{i\}$, all of these edges must be in colors other than the $i$th color and, again, our assumption that $K_n-e$ contains no color-$t$ copy of $K_{s_t}$ for any $t\in\{1,\hdots,k\}$ is violated.
        
    \end{proof}
    
    In the diagonal, two-color case, this says $\bigl||V(G_1)|-|V(G_2)|\bigr|<R\big(s-1,R(s)\big)-(s-2)$ and, since $R(s)=s$, this becomes the statement of \Cref{size lemma}.  In the multicolor case, this immediately strengthens \Cref{bound theorem}, the proof of which follows exactly along the same lines as \Cref{bound theorem}, imposing only the conditions of \Cref{size lemma 2}.
    
    \begin{thm}\label{formula theorem2}
        For $k \geq 2$ and $s_i \geq 3$ for every $i \in [k]$,
        $$
            m(s_1,\hdots,s_k) \leq \max\left\{\min_{i \in \{1,\hdots,k\}}\left\{\prod_{j=0}^{s_i-3}\left\lceil\frac{\big(n_i-2\big)-j}{s_i-2}\right\rceil\right\}\right\}
        $$
        
        where the maximum is taken over $n_1,\hdots,n_k \in \N$ satisfying the following:
        
        \begin{enumerate}[wide, labelindent=0pt]
            \item $n_i \geq s_i$ for every $i\in[k]$, 
            \item $n_1+\hdots+n_k\leq R(s_1,\hdots,s_k)+2(k-1)$, 
            \item For every ordered pair of distinct indices $(i,j)$ satisfying $n_i\geq n_j$, $$ n_i-n_j<R\left(s_i-1,R(s_1,\ldots,s_{i-1},s_{i+1},\ldots,s_k)\right)-(s_j-2).$$
        \end{enumerate}
        
    \end{thm}
    \begin{proof}
        For the $k$ extensions considered in \Cref{size lemma 2}, set $n_i=|V(G_i)|$. Condition (1) holds because each extension contains a color-$i$ copy of $K_{s_i}$, condition (2) follows from the overlap argument in the proof of \Cref{bound theorem}, and condition (3) is precisely \Cref{size lemma 2}. Applying \Cref{counting inequality} to each $G_i$, taking the minimum over $i$, and then maximizing over all feasible tuples gives the result.
    \end{proof}
    \begin{ex}\label{off diag ex}

        For each $i\in[k]$, write
        $$
            A_i=
            R\left(
                s_i-1,
                R(s_1,\ldots,s_{i-1},s_{i+1},\ldots,s_k)
            \right).
        $$
        Using \Cref{formula theorem2} and known values and upper bounds for
        small Ramsey numbers \cite{Radziszowski2026} gives the bounds recorded
        in \Cref{table:computation-details}.

        \begin{table}[hpt!]
            \centering
            \small
            \begin{tabularx}{\textwidth}{|c|c|C|c|c|}
                \hline
                Parameters
                & $R(s_1,\ldots,s_k)$
                & $(A_1,\ldots,A_k)$
                & Maximizing tuple
                & Bound
                \\\hline\hline
                $(3,3,3)$
                & $17$
                & $(6,6,6)$
                & $(7,7,7)$
                & $5$
                \\\hline
                $(3,3,3,3)$
                & $\leq62$
                & $(17,17,17,17)$
                & $(17,17,17,17)$
                & $15$
                \\\hline
                $(3,3,4)$
                & $30$
                & $(9,9,18)$
                & $(12,12,10)$
                & $10$
                \\\hline
                $(3,3,5)$
                & $\leq57$
                & $(14,14,40)$
                & $(23,23,13)$
                & $21$
                \\\hline
                $(3,3,6)$
                & $\leq91$
                & $(18,18,85)$
                & $(36,36,23)$
                & $34$
                \\\hline
            \end{tabularx}
            \caption{Ramsey-number inputs and one maximizing feasible tuple
            for each application of \Cref{formula theorem2}. The entries
            $62$, $57$, $91$, $40$, and $85$ are certified upper bounds;
            the remaining inputs are exact.}
            \label{table:computation-details}
        \end{table}
        For the $(3,3,4)$ case, we use the exact value $R(3,3,4)=R(4,3,3)=30$, established by Codish, Frank, Itzhakov, and Miller \cite{CodishFrankItzhakovMiller2016}.  Here and below, replacing an unknown Ramsey number by a certified upper bound only enlarges the feasible region in \Cref{formula theorem2}, and therefore yields a valid, possibly weaker, upper bound. In particular, the first of these inequalities recovers the exact value $m(3,3,3)=5$ proved by Sane and Wallis \cite{SaneWallis1988}.  For comparison with previously published finite-threshold values, see \Cref{table:known-multicolor-comparison} below.
        \begin{table}[hpt!]
            \centering
            \begin{tabularx}{\textwidth}{|c|C|c|}\hline
                Parameters
                & Previously published finite-threshold value
                & \Cref{formula theorem2}
                \\\hline\hline
                $(3,4)$ & $m(3,4)=1$ \cite{KhadzhiivanovNenov1978}
                        & $m(3,4)\leq 3$ \\\hline
                $(3,5)$ & $m(3,5)=4$ \cite{Pashov1984}
                        & $m(3,5)\leq 6$ \\\hline
                $(3,6)$ & $m(3,6)=2$ \cite{KhadshiivanovPashov1992}
                        & $m(3,6)\leq 8$ \\\hline
                $(3,3,3)$ & $m(3,3,3)=5$ \cite{SaneWallis1988}
                          & $m(3,3,3)\leq 5$ \\\hline
                $(3,3,4)$ & --- & $m(3,3,4)\leq 10$ \\\hline
            \end{tabularx}
            \caption{Comparison of several known finite-threshold multiplicities with the bounds supplied by \Cref{formula theorem2}. We are unaware of a previously published bound specific to $m(3,3,4)$.}
            \label{table:known-multicolor-comparison}
        \end{table}
        
    \end{ex}
\section{A Route for Improvement}
    
    We can extend the result of \Cref{bound theorem} (as well as the preceding results, through suitable modifications) using a stronger form of \Cref{counting inequality}.  The proof is similar.

    \begin{cor}\label{expanded counting inequality}
    
        Let $s \geq 3$ and let $H$ be a graph such that (i) $H$ contains a total of $t$ distinct copies of $K_s$ (i.e. subgraphs $H_1,\hdots,H_t$ each isomorphic to $K_s$ such that $H_i \neq H_j$ for $i \neq j$); and (ii) there exists a subgraph $G \leq H$ such that, for every copy $Q$ of $K_s$ in $H$, there is a unique edge $e\in E(G)$ satisfying $V(Q)\cap V(G)=V(e)$, and such that every edge of $G$ belongs to at least one copy of $K_s$ in $H$.  Then, 
        
        $$ 
            |E(G)| \leq t \leq |E(G)|\prod_{j=0}^{s-3}\left\lceil \frac{\left(|V(H)|-|V(G)|\right)-j}{s-2}\right\rceil.
        $$
        
    \end{cor}
    \begin{proof}
    
        Let $H$ be any graph satisfying $(i)$ and let $G \leq H$ satisfy $(ii)$. For each edge $e=uv\in E(G)$, let $C_{uv}$ denote the subset of vertices in $V(H)\setminus V(G)$ adjacent to both $u$ and $v$.  Let $H[C_{uv}]$ denote the subgraph of $H$ induced by $C_{uv}$.  There is then a bijection between the copies of $K_s$ in $H$ containing $e=uv$ and the copies of $K_{s-2}$ in $H[C_{uv}]$.  Moreover, $H[C_{uv}]$ is $K_{s-1}$-free, since a copy of $K_{s-1}$ in $H[C_{uv}]$, together with $u$ (or $v$, if you prefer), would be a copy of $K_s$ in $H$ containing no edge of $G$.  So, by Zykov's theorem \cite{Zykov1949}, $H[C_{uv}]$ contains at most the number of $K_{s-2}$'s in the Tur\'an graph $T(|V(H[C_{uv}])|,s-2)$.  But summing the number of $K_{s-2}$'s in $H[C_{uv}]$ over $e=uv\in E(G)$ yields $t$, as each copy of $K_s$ in $H$ contains exactly one edge of $G$ and no other vertices of $G$.  Moreover, $|E(G)|\leq t$, since each edge of $G$ belongs to at least one copy of $K_s$ in $H$.  Since $|C_{uv}|\leq |V(H)|-|V(G)|$, we then get
        $$|E(G)|\leq t \leq \sum_{uv\in E(G)}\prod_{j=0}^{s-3}\left\lceil\frac{|C_{uv}|-j}{s-2}\right\rceil\leq |E(G)|\prod_{j=0}^{s-3}\left\lceil\frac{|V(H)|-|V(G)|-j}{s-2}\right\rceil.$$
        
    \end{proof}

    To state the general form of the extended \Cref{bound theorem}, we require a definition.

    \begin{defi}\label{cleanly saturated coloring}
    
        Let $\mathbf{s}=(s_1,\ldots,s_k)$, let $G\leq K_n$ have no isolated vertices, and let
        $$
            \chi:E(K_n)\setminus E(G)\longrightarrow [k]
        $$
        be a $k$-edge-coloring.  For each $i\in[k]$, let $\chi_i$ denote the extension of $\chi$ to $K_n$ obtained by assigning color $i$ to every edge of $G$.  We say that $\chi$ is \emph{cleanly $\mathbf{s}$-saturated} with respect to $G$ if the following conditions hold:
    
        \begin{enumerate}
            \item $\chi$ contains no color-$i$ copy of $K_{s_i}$ for any
            $i\in[k]$;
    
            \item for every $i\in[k]$ and every edge $e\in E(G)$, the edge $e$ belongs to at least one color-$i$ copy of $K_{s_i}$ in $\chi_i$;
            \item for every $i\in[k]$ and every color-$i$ copy
            $Q\cong K_{s_i}$ in $\chi_i$, there exists an edge $e\in E(G)$ such
            that
            $$
            V(Q)\cap V(G)=V(e).
            $$
        \end{enumerate}
    
    \end{defi}
    \Cref{cleanly saturated coloring} formalizes the property of the coloring considered in \Cref{last edge lemma}.  In this language, \Cref{last edge lemma} says that if $R(s_1,\hdots,s_k)=n$, then there is a $K_2$ subgraph of $K_n$ (i.e. an edge $e$) and a $k$-edge-coloring $\chi$ of $E(K_n) \setminus \{e\}$ that is cleanly $(s_1,\ldots,s_k)$-saturated with respect to this edge.  Using this coloring, we were able to produce the bound on $m(s_1,\hdots,s_k)$ in \Cref{bound theorem}.  We will now create a similar bound for larger subgraphs $G\leq K_n$ and cleanly $\mathbf{s}$-saturated colorings of $E(K_n)\setminus E(G)$.  Recall also that a \emph{matching} of a graph $G$ is a set of pairwise non-incident edges, a \emph{maximum matching} is a matching with the largest possible number of edges, and the \emph{matching number} of $G$, denoted $\nu(G)$, is the size of a maximum matching.
    \begin{thm}\label{bound theorem m version}
    
        Let $n=R(s_1,\ldots,s_k)$, where $s_i \geq 3$ for $1 \leq i \leq k$. Suppose that there exist a graph $G\leq K_n$ and a coloring $\chi$ of $E(K_n)\setminus E(G)$ that is cleanly $(s_1,\ldots,s_k)$-saturated with respect to $G$.  Then, 
        
        \begin{equation}\label{ourbound m version}         
            m(s_1,\ldots,s_k) \leq \max_{\substack{(n_1,\ldots,n_k)\in\mathbb N^k\\s_i-2\leq n_i\leq n-|V(G)|\text{ for every }i\in[k]\\ n_1+\cdots+n_k \leq \bigl(n-|V(G)|\bigr)\min\{k,\nu(G)\}}}\min_{1\leq i\leq k}\left\{|E(G)|\prod_{j=0}^{s_i-3}\left\lceil\frac{n_i-j}{s_i-2}\right\rceil\right\}
        \end{equation}
        
        Likewise, we have the weaker bound
        
        \begin{equation}\label{ourbound m version weak case}     
            m(s_1,\hdots,s_k)\leq \max_{1\leq i\leq k}|E(G)|\prod_{j=0}^{s_i-3}\left\lceil \frac{\left\lfloor \frac{\min\{k,\nu(G)\}(n-|V(G)|)}{k} \right\rfloor-j}{s_i-2}\right\rceil.
        \end{equation}
        
    \end{thm}
    \begin{proof}
    
        Let $n=R(s_1,\ldots,s_k)$ and suppose that there exist a graph $G\leq K_n$ and a coloring $\chi$ of $E(K_n)\setminus E(G)$ that is cleanly $(s_1,\ldots,s_k)$-saturated with respect to $G$.  Consider the $k$ possible extensions $\chi_1,\hdots,\chi_k$ of this coloring to $K_n$ obtained by coloring all of the edges $E(G)$ with color $i$. Each $\chi_i$ must contain some number $t_i \geq m(s_1,\hdots,s_k)\geq 1$ of copies of $K_{s_i}$ by \Cref{ramsey def,crit mult def}.  Moreover, these are the only occurrences of such monochromatic subgraphs in $\chi_i$:  a color-$j$ copy of $K_{s_j}$ with $j\neq i$ cannot appear in $\chi_i$, as such a copy cannot contain an edge of $G$ (because every edge of $G$ has color $i$) and hence would already occur in the target-free coloring $\chi$.  For each $\chi_i$, let $H_i$ be the corresponding subgraph of $K_{n}$ consisting of the union of the $t_i$ monochromatic copies of $K_{s_i}$ contained in $\chi_i$. Note that each $H_i$ contains $G$ itself, as \Cref{cleanly saturated coloring} implies that every edge of $G$ belongs to at least one monochromatic copy of $K_{s_i}$ in $\chi_i$.  Likewise, by \Cref{cleanly saturated coloring}, each $H_i$ satisfies the conditions of \Cref{expanded counting inequality} with respect to $G$. So, taking the minimum over $i\in [k]$, we get 
        
        \begin{equation}\label{hbound}
            m(s_1,\hdots,s_k)\leq \min_{1 \leq i \leq k}|E(G)|\prod_{j=0}^{s_i-3}\left\lceil \frac{\left(|V(H_i)|-|V(G)|\right)-j}{s_i-2}\right\rceil.
        \end{equation}
        
        Note now that $s_i-2 \leq |V(H_i)|-|V(G)|\leq n-|V(G)|$, since $H_i \leq K_n$ for the latter inequality and, for the former, since $H_i$ contains at least one $K_{s_i}$ that has but one edge in $G$ by \Cref{cleanly saturated coloring}.  Note also that if $J\subseteq [k]$ is such that a vertex $v \in V(K_n)$ satisfies $v \in V(H_j)\setminus V(G)$ for all $j \in J$, then this imposes that there must be at least $|J|$ pairwise non-incident edges in $G$. To see this, for each $j\in J$, choose a monochromatic $K_{s_j}$ containing $v$ and let $e_j$ be its unique edge in $G$. Were a pair of such edges $e_{j_1}$ and $e_{j_2}$ adjacent in $G$, then their common incident vertex $w \in V(G)$ would imply that the edge $vw\in E(H_{j_1}) \cap E(H_{j_2})$ must be colored both $j_1$ and $j_2$. But $vw\notin E(G)$, so its color is fixed by $\chi$, a contradiction.  Accordingly, such a subset $J$ must satisfy $|J| \leq \nu(G)$ and, likewise, $|J| \leq k$.  Hence, each vertex $v \in V(K_n)\setminus V(G)$ belongs to at most $\min\left\{k,\nu(G)\right\}$ of the graphs $H_1,\hdots,H_k$.  Accordingly, we get
        
        \begin{equation}\label{sumbound}
            \sum_{i=1}^k \big( |V(H_i)|-|V(G)| \big)\leq \left(n-|V(G)|\right) \min\left\{k,\nu(G)\right\}.
        \end{equation}

        The tuple $(n_1,\hdots,n_k)$ defined by $n_i=|V(H_i)|-|V(G)|$ is feasible by the preceding inequalities. Hence, \eqref{hbound} gives the first desired claim:
        
        $$
            m(s_1,\ldots,s_k) \leq \max_{\substack{(n_1,\ldots,n_k)\in\mathbb N^k\\s_i-2\leq n_i\leq n-|V(G)|\text{ for every }i\in[k]\\ n_1+\cdots+n_k \leq \bigl(n-|V(G)|\bigr)\min\{k,\nu(G)\}}}\min_{1\leq i\leq k}\left\{|E(G)|\prod_{j=0}^{s_i-3}\left\lceil\frac{n_i-j}{s_i-2}\right\rceil\right\}.
        $$

        For the second bound, \eqref{sumbound} implies that, for at least one $i \in [k]$, the quantity $|V(H_i)|-|V(G)|$ must satisfy
        
        $$
            |V(H_i)|-|V(G)| \leq \left\lfloor \frac{\min\{k,\nu(G)\}(n-|V(G)|)}{k} \right\rfloor,
        $$
        as the contrary would yield the contradiction
        
        \begin{align*}
            \sum_{i=1}^k \big( |V(H_i)|-|V(G)| \big)
            &\geq k\left(\left\lfloor \frac{\min\{k,\nu(G)\}(n-|V(G)|)}{k} \right\rfloor+1\right)\\
            &> \min\{k,\nu(G)\}(n-|V(G)|).
        \end{align*}
        
        Since the relevant product is nondecreasing in $|V(H_i)|-|V(G)|$, \eqref{hbound} yields 
        $$
            m(s_1,\hdots,s_k)\leq \max_{1 \leq i \leq k}|E(G)|\prod_{j=0}^{s_i-3}\left\lceil \frac{\left\lfloor \frac{\min\{k,\nu(G)\}(n-|V(G)|)}{k} \right\rfloor-j}{s_i-2}\right\rceil.
        $$
        
    \end{proof}

    \begin{remark}
    
        In the case where $G=K_2$, we get $|E(G)|=1$, $|V(G)|=2$, and $\nu(G)=1$, producing (via the weak bound \eqref{ourbound m version weak case})
        
        $$
            m(s_1,\hdots,s_k) \leq \max_{1\leq i \leq k}\prod_{j=0}^{s_i-3}\left\lceil \frac{\left\lfloor \frac{(R(s_1,\hdots,s_k)-2)}{k} \right\rfloor-j}{s_i-2}\right\rceil 
        $$
        Thus, \Cref{bound theorem m version} is an extension of \Cref{bound theorem}.
        
    \end{remark}

    \begin{remark}
    
        The matching-number constraint is most favorable when $\nu(G)=1$, the smallest possible value for a nonempty graph $G$. In this case, the sets $V(H_i)\setminus V(G)$ are pairwise disjoint. Moreover, if $G$ is a forest without isolated vertices, then $\nu(G)=1$ if and only if $G$ is a star. This does not by itself guarantee the best numerical bound, since increasing the size of the star trades a smaller value of $n-|V(G)|$ against the larger prefactor $|E(G)|$ in \eqref{ourbound m version}.
        
    \end{remark}
    
\section{Concluding Remarks}\label{conclude}

    In \Cref{last edge lemma}, we show that there is a partial coloring of $K_{R(s_1,\ldots,s_k)}$ that is cleanly $(s_1,\ldots,s_k)$-saturated with respect to a single uncolored edge.  Combined with the counting argument, this yields the bounds in \Cref{bound theorem,formula theorem2}. In the two-color diagonal case, these bounds improve the elementary probabilistic estimate for $5\leq s\leq8$, while in the multicolor case \Cref{formula theorem2} recovers $m(3,3,3)=5$ and yields $m(3,3,4)\leq10$.

    \Cref{bound theorem m version} shows more generally that producing a partial coloring that is cleanly $(s_1,\ldots,s_k)$-saturated with respect to its uncolored subgraph yields a bound on $m(s_1,\ldots,s_k)$. This provides a framework for further improvements. In particular, it remains to find a cleanly saturated coloring with respect to a graph $G\neq K_2$ that improves the numerical bounds obtained from the one-edge construction.

\bibliographystyle{amsplain}
\bibliography{ref}
\section*{Statements and Declarations}
\noindent\textbf{Conflict of Interest Statement}: The authors declare none. \\\textbf{Funding Statement}: This work received no specific grant from any funding agency, commercial or not-for-profit sectors.  \\\textbf{Data Availability Statement}: Data sharing is not applicable to this article, as no datasets were generated or analyzed during the current study.
\end{document}